\documentclass[10pt]{amsart}
\title{A note on L-packets and abelian varieties over local fields}
\usepackage[T1]{fontenc}
\usepackage[utf8]{inputenc}
\usepackage{latexsym}   
\usepackage{amssymb}    
\usepackage{amsmath}    
\usepackage{amsthm}     
\usepackage{stmaryrd}
\usepackage{hyperref}
\usepackage{geometry}
\usepackage{amsrefs,amsxtra,comment,xspace}
\usepackage[all]{xy}
\usepackage[british]{babel}
\date{\today}
\numberwithin{equation}{section}
\newtheorem{theorem}{Theorem}[section]

\newtheorem{lemma}[theorem]{Lemma}
\newtheorem{proposition}[theorem]{Proposition}
\newtheorem{corollary}[theorem]{Corollary}

\theoremstyle{definition}

\newtheorem{example}[theorem]{Example}

\def\charpoly{P}
\def\minpoly{M}

\usepackage{diagrams}
\newarrow{to}---->
\newarrow{dotsto}....>
\newarrow{mto}|--->
\newarrow{line}-----
\newarrow{impto}===={=>}
\newarrow{correspto}..+.>
\newarrowtail{<=}\Rightarrow\Leftarrow\Uparrow\Downarrow
\newarrow{bboth}{<=}==={=>}
\newarrow{inject}{hooka}---{vee}
\newarrow{surject}----{>>}

\def\cross{\times}

\DeclareMathOperator{\GL}{GL}

\DeclareMathOperator{\SL}{SL}
\DeclareMathOperator{\Sp}{Sp}
\DeclareMathOperator{\SO}{SO}
\newcommand{\Gm}{\mathbb{G}_{\hskip-1pt\textbf{m}}}
\DeclareMathOperator{\GSp}{GSp}
\DeclareMathOperator{\PGSp}{PGSp}
\DeclareMathOperator{\Spin}{Spin}
\DeclareMathOperator{\GSpin}{GSpin}

\def\ra{\rightarrow}

\DeclareMathOperator{\aut}{Aut}

\def\ff{{\mathbb F}}

\def\calo{{\mathcal O}}

\def\calx{{\mathcal X}}
\DeclareMathOperator{\frob}{Fr}

\newcommand{\ang}[1]{{{\langle #1 \rangle}}}

\newcommand{\rest}[1]{|_{#1}}
\newcommand{\st}[1]{\{#1\}}

\newenvironment{alphabetize}{\begin{enumerate}

}{\end{enumerate}}

\newcommand{\dual}[1]{{\check{#1}}}
\newcommand{\Lgroup}[1]{{\,^L\hskip-1pt{#1}}}
\newcommand{\dualgroup}[1]{{\dual{#1}}}

\DeclareMathOperator\Ind{Ind}
\DeclareMathOperator\Lie{Lie}

\DeclareMathOperator\Gal{Gal}
\DeclareMathOperator\End{End}

\DeclareMathOperator\gal{Gal}

\def\cross{\times}

\def\der{_{\rm der}}
\def\ad{_{\rm ad}}

\def\integ{{\mathbb Z}}
\def\Q{{\mathbb Q}}
\def\rat{\Q}
\def\tensor{{\otimes}}

\def\cpx{\mathbb{C}}
\def\real{\mathbb{R}}

\def\twiddle{\sim}

\newcommand\ceq{{\, := \,}}
\newcommand\tq{{\, \vert \, }}

\newcommand\iso{{\ \cong\ }}

\newcommand{\abs}[1]{{\vert #1 \vert}}
\DeclareMathOperator{\Fr}{Fr}

\DeclareMathOperator{\zendo}{ZEnd^0}
\DeclareMathOperator{\Endo}{End^0}
\makeatletter
\newcommand{\labitem}[2]{%
\def\@itemlabel{\textbf{#1}}
\item
\def\@currentlabel{#1}\label{#2}}
\makeatother

\author{Jeffrey D. Achter}
\address{Colorado State University, Fort Collins, CO}
\email{achter@math.colostate.edu}
\urladdr{http://www.math.colostate.edu/~achter}
\thanks{JDA was partially supported by a grant from the Simons Foundation (204164).}

\author{Clifton Cunningham}
\address{University of Calgary, Canada}
\email{cunning@math.ucalgary.ca}
\thanks{CC was partially supported by NSERC (DG696158) and PIMS}
\subjclass[2010]{11G10 (primary), 11S37 (secondary)}
\keywords{abelian varieties, good reduction, local fields, L-packets, admissible representations}

\begin{document}

\begin{abstract}
A polarized abelian variety $(X,\lambda)$ of dimension $g$ and good reduction over a local field $K$ determines an admissible representation of $\GSpin_{2g+1}(K)$.  We show that the restriction of this representation to $\Spin_{2g+1}(K)$ is reducible if and only if $X$ is isogenous to its twist by the quadratic unramified extension of $K$. When $g=1$ and $K = \rat_p$, we recover the well-known fact that the admissible $\GL_2(K)$ representation attached to an elliptic curve $E$ with good reduction is reducible upon restriction to $\SL_2(K)$ if and only if $E$ has supersingular reduction.
\end{abstract}

\maketitle

\section*{Introduction}

Consider an elliptic curve $E/\rat_p$ with good reduction.
Let $\pi_E$ be the unramified principal series representation
of  $\GL_2(\rat_p)$ with the same Euler factor as $E$.
Although $\pi_E$ is irreducible,
the restriction of $\pi_E$ from $\GL_2(\rat_p)$ to its derived group,
$\SL_2(\rat_p)$, need not be irreducible.  In fact,
it is not hard to show that $\pi_E\rest{\SL_2(\rat_p)}$
is reducible if and only if the reduction of $E$ is supersingular,
see \cite{Anand-Prasad:SL2}*{2.1} for example.

This note generalizes the observation above,
as follows.
Let $K$ be a non-Archimedean local field with finite residue field
and let $(X,\lambda)$ be a polarized abelian variety over $K$ of dimension $g$ with good reduction.
Fix a rational prime $\ell$ invertible in the residue field of $K$.
Then the associated Galois representation on the rational $\ell$-adic Tate module of $X$ takes values in
$\GSp(V_\ell X, \ang{\cdot,\cdot}_\lambda) \iso \GSp_{2g}(\rat_\ell)$.
The eigenvalues of the image of Frobenius under this unramified representation determine an irreducible principal series representation $\pi_{X,\lambda}$ of $\GSpin_{2g+1}(K)$
with the same Euler factor as $X$.
Note that the dual group to $\GSpin_{2g+1}$ is $\GSp_{2g}$;
note also that $\GSpin_{3} \iso \GL_2$ and $\GSpin_{5} \iso \GSp_4$, accidentally.
In this note we show that the restriction of $\pi_{X,\lambda}$ from $\GSpin_{2g+1}(K)$
to its derived group $\Spin_{2g+1}(K)$ is reducible if and only if
$X$ is isogenous to its twist by the quadratic unramified extension of $K$.

Furthermore, we identify the Langlands parameter $\phi_{X,\lambda}$ for $\pi_{X,\lambda}$
and then show that the corresponding L-packet $\Pi_{X,\lambda}$ contains
the equivalence class of $\pi_{X,\lambda}$ only.
Then we show that we can detect when $X$ is isogenous to its quadratic unramified twist directly from the local L-packet $\Pi^\text{der}_{X,\lambda}$ determined by
transferring the Langlands parameter $\phi_{X,\lambda}$ to the derived group $\Spin_{2g+1}(K)$ of $\GSpin_{2g+1}(K)$.

\smallskip
\noindent{\bf Acknowledgements.}
We thank V. Vatsal, P. Mezo and U. K. Anandavardhanan for helpful conversations, and the referee for a careful reading and useful suggestions.

\tableofcontents

\section{Abelian varieties}\label{sec:even}

In this section, we collect some useful facts about abelian varieties, especially over finite fields.  Many of the attributes discussed here are isogeny invariants.
We write $X\twiddle Y$ if $X$ and $Y$ are isogenous abelian varieties, and $\Endo(X)$ for the endomorphism algebra $\End(X)\tensor_\integ\rat$ of $X$.

\subsection{Base change of abelian varieties}

Let $X/\ff_q$ be an abelian variety of dimension $g$.
Associated to it are the characteristic polynomial
$\charpoly_{X/\ff_q}(T)$ and minimal polynomial $\minpoly_{X/\ff_q}(T)$ of
Frobenius.  Then $\charpoly_{X/\ff_q}(T) \in \integ[T]$ is monic of degree
$2g$, and $\minpoly_{X/\ff_q}(T)$ is the radical of $\charpoly_{X/\ff_q}(T)$.

The isogeny class of $X$ is completely
determined by $\charpoly_{X/\ff_q}(T)$ \cite{Tate:endff}.  It is thus possible to
detect from $\charpoly_{X/\ff_q}(T)$ whether $X$ is simple, but even easier to
decide if $X$ is isotypic, which is to say, isogenous to the self-product of a simple abelian variety.
Indeed, let $\zendo(X) \subset \Endo(X)$ be the center of the
endomorphism algebra of $X$.  Then
\begin{equation}
\label{eqzend}
\zendo(X) \iso \rat[T]/\left(\minpoly_{X/\ff_q}(T)\right),
\end{equation}
and $X$ is isotypic if and only if $\minpoly_{X/\ff_q}(T)$ is irreducible.
While it is possible for a simple abelian variety to become reducible
after extension of scalars of the base field, isotypicality is
preserved by base extension (see \cite{Oort:2005survey}*{Claim 10.8} for example).

For a monic polynomial $g(T) = \prod_{1 \le j \le N}(T-\tau_j)$ and
a natural number $r$, set
$
g^{(r)}(T) = \prod_{1 \le j \le N}(T-\tau_j^r).
$
It is not hard to check that
\[
\charpoly_{X/\ff_{q^r}}(T) = \charpoly^{(r)}_{X/\ff_q}(T).
\]

\begin{lemma}
\label{lemdimzend}
Suppose $X/\ff_q$ is isotypic, and let $\ff_{q^r}/\ff_q$ be a finite
extension.  Let $Y$ be a simple factor of $X_{\ff_{q^r}}$.
Then there exists some $m|r$ such that
\begin{align*}
\minpoly_{X/\ff_{q}}^{(r)}(T) &= \minpoly_{Y/\ff_{q^r}}(T)^m\\
\intertext{and }
\dim \zendo(X) &= m \dim \zendo(X_{\ff_{q^r}}).
\end{align*}
\end{lemma}

\begin{proof}
Write $X_{\ff_{q^r}} \sim Y^n$ with $Y$ simple.
Then we have two
different factorizations of $\charpoly_{X/\ff_{q^r}}(T)$:
\begin{align*}
\charpoly_{X/\ff_{q^r}}(T) &= (\minpoly_{X/\ff_q}^{(r)}(T))^d \\
\charpoly_{X/\ff_{q^r}}(T) &= (\minpoly_{Y/\ff_{q^r}}(T))^e.
\end{align*}
Since $\minpoly_{Y/\ff_{q^r}}(T)$ is irreducible (and all polynomials
considered here are monic), there exists some integer $m$ such that
\[
\minpoly_{X/\ff_q}^{(r)}(T) = \minpoly_{Y/\ff_{q^r}}(T)^m.
\]
Note that
\[
m = \frac{\deg \minpoly_{X/\ff_q}(T)}{\deg \minpoly_{Y/\ff_{q^r}}(T)} =
[\zendo(X):\zendo(X_{\ff_{q^r}})].
\]
Let $\tau$ be a root of $\minpoly_{X/\ff_q}(T)$.
Then $\tau^r$ is a root of $\minpoly^{(r)}_{X/\ff_q}(T)$,
and thus of $\minpoly_{Y/\ff_q}(T)$; and the inclusion of fields
$\zendo(X_{\ff_{q^r}}) \subseteq \zendo(X)$
is isomorphic to the inclusion of fields
$\rat(\tau^r) \subseteq \rat(\tau)$, under \eqref{eqzend}.
In particular, $m =
[\rat(\tau):\rat(\tau^r)]$.  Since $\tau$ satisfies the equation
$S^r - \tau^r$ over $\rat(\tau^r)$, its degree over
$\rat(\tau^r)$ divides $r$.
\end{proof}

\subsection{Even abelian varieties}
\label{subsec:even}

Call an abelian variety $X/\ff_q$ {\em even} if its characteristic polynomial is even:
\[
\charpoly_{X/\ff_q}(T) = \charpoly_{X/\ff_q}(-T).
\]

If $X$ is simple, then it admits a unique nontrivial quadratic twist
$X'/\ff_q$.  For an arbitrary $X/\ff_q$, let $X'/\ff_q$ be the
quadratic twist associated to the cocycle
\begin{diagram}
\gal(\ff_q) & \rto & \aut(X) \\
\frob_q & \rmto & [-1],
\end{diagram}
corresponding to a nontrivial quadratic twist of all simple factors of
$X$.

For future use, we record the following elementary observation:
\begin{lemma}
\label{lemeven}
Let $X/\ff_q$ be an abelian variety.  Then $X$ is even if and only if $X$ and $X'$ are isogenous.
\end{lemma}

\begin{proof}
Use the (canonical, given our construction) isomorphism
$X_{\ff_{q^2}} \iso X'_{\ff_{q^2}}$
to identify $V_\ell X$ and $V_\ell X'$.
Then one knows
(see \cite{Serre-Tate:good}*{p.506} for example) that $\rho_{X'/\ff_q}(\Fr_q) = - \rho_{X/\ff_q}(\Fr_q)$, and thus that
\[
\charpoly_{X'/\ff_q}(T) = \charpoly_{X/\ff_q}(-T).
\]
The asserted equivalence now follows from Tate's theorem \cite{Tate:endff}*{Th 1}.
\end{proof}

To a large extent, evenness of $X$ is captured by the behavior of the
center of $\Endo(X)$ upon quadratic base extension.

\begin{lemma}
\label{lemevenzend}
If $X/\ff_q$ is even, then
\[
\dim \zendo(X) = 2\dim \zendo(X_{\ff_{q^2}}).
\]
\end{lemma}

\begin{proof}
Suppose $X/\ff_q$ is even. Then the multiset $\st{\tau_1, \cdots, \tau_{2g}}$
of eigenvalues of Frobenius of $X$ is stable under multiplication by $-1$, and in
particular the set of distinct eigenvalues of Frobenius is stable under
multiplication by $-1$.  Moreover, this action has no fixed points;
and thus $\st{\tau_1^2, \cdots, \tau_{2g}^2}$, the set of
eigenvalues of $X/\ff_{q^2}$, has half as many distinct elements as
the original set.  The claim now follows from characterization
\eqref{eqzend} of $\zendo(X)$.
\end{proof}

The converse is almost true.

\begin{proposition}
\label{propevenzend}
Suppose $X$ is isotypic.  Then $X$ is even if and only if
\[\dim\zendo(X) = 2 \dim \zendo(X_{\ff_{q^2}}).\]
\end{proposition}

\begin{proof}
Suppose $\dim \zendo(X) = 2 \dim \zendo(X_{\ff_{q^2}})$ and let $Y$ be
a simple factor of $X_{\ff_{q^2}}$.
By Lemma \ref{lemdimzend},
\begin{equation}
\label{eqmamy2}
\minpoly^{(2)}_{X/\ff_q}(T) = \minpoly_{Y/\ff_{q^2}}(T)^2.
\end{equation}
Factor the minimal
polynomials of $X$ and $Y$ as
\begin{align*}
\minpoly_{X/\ff_q}(T) &= \prod_{1 \le j \le 2h}(T-\tau_j) \\
\minpoly_{Y/\ff_{q^2}}(T) &= \prod_{1 \le j \le h}(T-\beta_j).
\end{align*}
By \eqref{eqmamy2}, we may order the roots of $\minpoly_{X/\ff_q}(T)$ so
that, for each $1 \le j \le h$, we have
\[
\tau_j^2 = \tau_{h+j}^2 = \beta_j,
\]
so that $\tau_{h+j} = \pm \tau_j$.  In fact, $\tau_{h+j} =
-\tau_j$; for otherwise, $\minpoly_{X/\ff_q}(T)$ would have a repeated
root, which contradicts the known semisimplicity of Frobenius.
 Now,
$\charpoly_{X/\ff_q}(T) = \minpoly_{X/\ff_q}(T)^d$ for some $d$.  The multiset of
eigenvalues of Frobenius of $X$ is thus stable under multiplication by
$-1$, and $X/\ff_q$ is even.
\end{proof}

Note that evenness is an assertion about the {\em multiset} of eigenvalues
of Frobenius, while the calculation of $\dim \zendo(X_{\ff_{q^e}})$
only detects the {\em set} of eigenvalues.
Consequently, if one drops the isotypicality assumption in
Proposition~\ref{propevenzend}, it is easy to write down examples
of abelian varieties which are not even but satisfy the criterion
on dimensions of centers of endomorphism rings.

\begin{example}
Let $E/\ff_q$ be an ordinary elliptic curve; then $E$ is not isogenous to $E'$ over $\ff_q$ but
$\Endo(E) \iso \Endo(E') \iso L$, a quadratic imaginary field.  Set $X = E \cross E \cross E'$.  Then $X$ is not even, since $X' \iso E' \cross E' \cross E$, but $\zendo(X) \iso L\cross L$ while $\zendo(X_{\ff_{q^2}}) \iso L$.
Thus, $X/\ff_q$ satisfies the dimension criterion of Proposition~\ref{propevenzend} but is not even.
\end{example}

\begin{example}
Consider a supersingular elliptic curve $E/\ff_q$, where $q$ is an odd power of the prime $p$.  Then $\Endo(E) \iso \rat(\sqrt{-p})$, while $\Endo(E_{\ff_{q^2}})$ is the quaternion algebra ramified at $p$ and $\infty$.  In particular, $\zendo(E)$ is a quadratic imaginary field, while $\zendo(E_{\ff_{q^2}}) \iso \rat$.  Therefore, $E/\ff_q$ is even.
\end{example}
\begin{example}
In contrast, if $X/\ff_q$ is an absolutely simple ordinary abelian variety, then $\Endo(X) = \Endo(X_{\ff_{q^2}})$.
(This is a consequence of  \cite{Waterhouse}*{Thm. 7.2}, which unfortunately omits the necessary hypothesis of absolute simplicity.)
\end{example}

\begin{example}
Now consider an arbitrary abelian variety $X/\ff_q$ and its preferred quadratic twist $X'$.  Then the sum $X\cross X'$ is visibly isomorphic to its own quadratic twist, and thus even.
\end{example}


\begin{example}
Let $X/\ff_q$ be an abelian variety of dimension $g$.  Suppose there is an integer $N\ge 3$, relatively prime to $q$, such that $X[N](\ff_q) \iso (\integ/N)^{2g}$.  Then $X$ is not even.
Indeed, if an abelian variety $Y$ over a field $k$ has maximal $k$-rational $N$-torsion for $N \ge 3$ and $N$ is invertible in $k$, then $\Endo(Y) \iso \Endo(Y_{\bar k})$ \cite{Silverberg}*{Thm.\ 2.4}.  By the criterion of Lemma \ref{lemevenzend}, if $X/\ff_q$ satisfies the hypotheses of the present lemma, then $X$ cannot be even.

\end{example}

\subsection{Abelian varieties over local fields}
\label{subsec:evenlocal}

Now  let $K$ be a local field with residue field $\ff_q$
and let $X/K$ be an abelian variety with good reduction $X_0/{\ff_q}$.  As in \ref{subsec:even}, we define a canonical quadratic twist $X'$ of $X$, associated to the unique nontrivial character
\begin{diagram}
\gal(\bar K/K) & \rto & \gal(K^{{\rm unram}}/K) & \rto & \st{[\pm 1]} \subset \aut(X).
\end{diagram}

\begin{proposition}
\label{lemevenlocal}
Let $X/K$ be an abelian variety with good reduction $X_0/{\ff_q}$.  The following are equivalent:
\begin{alphabetize}
\item $X$ and $X'$ are isogenous;
\item $X_0/{\ff_q}$ and $X'_0/{\ff_q}$ are isogenous;
\item $X_0/{\ff_q}$ is even.
\end{alphabetize}
\end{proposition}

\begin{proof}
By hypothesis, $X$ spreads out to an abelian scheme $\calx/\calo_K$ (its N\'eron model) with special fibre $X_0/{\ff_q}$; the automorphism $[-1] \in \End(X)$ extends to an automorphism of $\calx$ and the corresponding twist $\calx'$ has generic and special fibers $X'$ and $(X_0)'/{\ff_q}$,
respectively.  This compatibility explains the equivalence of (a) and (b); the equivalence of (b) and (c) is  Lemma \ref{lemeven}.
\end{proof}

Call $X/K$ \emph{even} if $X$ has good reduction and
satisfies any of the equivalent statements in Proposition \ref{lemevenlocal}.

\section{L-packets attached to abelian varieties}
\label{sec:LpiX}

\subsection{Polarizations}\label{sec:polar}
Let $X/k$ be an abelian variety over an arbitrary field $k$.
Let $\lambda$ be a polarization on $X$, {\it i.e.}, a symmetric isogeny $X
\ra \hat X$ which arises from an ample line bundle on $X$.
Fix a rational prime $\ell$ invertible in $k$.
The polarization $\lambda$ on $X$ induces a nondegenerate skew-symmetric pairing $\ang{\cdot,\cdot}_\lambda$ on
the Tate module $T_\ell X$ and on the rational Tate module $V_\ell X$.
Let $\GSp(V_\ell X,\ang{\cdot,\cdot}_\lambda)$
be the group of symplectic similitudes of $V_\ell X$ with respect to
this pairing;
 note that $\GSp(V_\ell X,\ang{\cdot,\cdot}_\lambda)$ comes with a representation
$
r_{\lambda,\ell}: \GSp(V_\ell X,\ang{\cdot,\cdot}_\lambda) \hookrightarrow \GL(V_\ell X).
$
Let $\rho_{X,\ell} : \Gal(k) \to \GL(V_\ell X)$ be the representation on the rational Tate module and let
$
\rho_{\lambda,\ell} : \Gal(k) \to \GSp(V_\ell X,\ang{\cdot,\cdot}_\lambda)
$
be the continuous homomorphism such that $\rho_{X,\ell}  = r_{\lambda,\ell} \circ \rho_{\lambda,\ell}$.
\begin{equation}\label{rell}
\xymatrix{
\Gal({\bar k}/k) \ar[rr]^{\rho_{X,\ell}}  \ar[rd]_{\rho_{X,\lambda,\ell}} && \GL(V_\ell X) \\
& \GSp(V_\ell X,\ang{\cdot,\cdot}_\lambda) \ar[ur]_{r_{\lambda,\ell}} & \\
}
\end{equation}

\subsection{Admissible representations attached to abelian varieties with good reduction}\label{sec:piX}

Let $K$ be a local field.
Fix a rational prime $\ell$ invertible in the residue field of $K$, and thus in $K$.
It will be comforting, though not even remotely necessary, to fix an isomorphism $\bar\rat_\ell \iso \cpx$. We will indicate the corresponding complex-valued versions of $\rho_{X,\ell}$, $\rho_{\lambda,\ell}$ and $r_{\lambda,\ell}$ from Section~\ref{sec:polar} by eliding the subscript $\ell$.

In the rest of the paper we will commonly employ the notation $G \ceq \GSpin_{2g+1}$; note that the dual group to $G$ is $\dual{G} = \GSp_{2g}$.
The derived group $G\der= \Spin_{2g+1}$, which is semisimple and simply connected, will play a role below, as will its dual $\dual{G}\ad = \PGSp_{2g}$, which is of adjoint type.

\begin{proposition}\label{prop:piX}
Let $X/K$ be an abelian variety of dimension $g$ with with good reduction
and let $\lambda$ be a polarization on $X$.
There is an irreducible unramified principal series representation $\pi_{X,\lambda}$ of $\GSpin_{2g+1}(K)$, unique up to equivalence, such that
\[
L(z,\rho_{X}) = L(z,\pi_{X,\lambda}, r_\lambda).
\]
Moreover, $\abs{\ }_{K}^{-\frac{1}{2}}\otimes \pi_{X,\lambda}$ is unitary.
\end{proposition}

\begin{proof}
This is a very small and well-known part of the local Langlands correspondence for $G = \GSpin_{2g+1}$ over $K$ which,
in this case, matches unramified principal series representations of $G(K)=\GSpin_{2g+1}(K)$ with unramified Langlands parameters taking values in $\dual{G}(\cpx) = \GSp_{2g}(\cpx)$.
For completeness and to introduce notation for later use, we include the details here.

We begin by describing $L(z,\rho_{X})$.
By \cite{Serre-Tate:good}, the Galois representation $\rho_{X,\ell}$ is unramified
and the characteristic polynomial of $\rho_{X,\ell}(\Fr_q)$ has
rational coefficients.
Accordingly, the Euler factor for $\rho_{X,\ell}$ takes the form
\[
L(s,\rho_{X}) = \frac{(q^s)^{2g}}{\charpoly_{X_0/{\ff_q}}(q^{s})}.
\]
Let $\{\tau_1, \ldots, \tau_{2g}\}$ be the (complex) roots of $\charpoly_{X_0/{\ff_q}}(T)$.
Also by \cite{Serre-Tate:good}, the $\ell$-adic realization
$\rho_{X,\ell}(\Fr_q)\in \GL(V_\ell X)$ of the Frobenius endomorphism
of $X$ is semisimple of weight $1$, so each eigenvalue satisfies
$\abs{\tau_j} = \sqrt{q}$.
Label the roots in such a way that, for each $1 \le j \le g$, we have
$\tau_{g+j} = q\tau_j^{-1}$; and $\tau_j = \sqrt q e^{2\pi i
  \theta_j}$, where $1 > \theta_1 \ge \theta_2 \cdots \ge \theta_g \ge
0$.

Let $T$ be a $K$-split maximal torus in $G$;
let $\dual{T}$ be the dual torus.
Then the Lie algebra of the torus $\dual{T}(\cpx)$
may be identified with $X^*(T) \otimes \cpx$
through the function
\[
\exp: X^*(T)\otimes \cpx \to \dual{T}(\cpx)
\]
defined by $\dual{\alpha}(\exp(x)) = e^{2\pi i \langle \dual{\alpha}, x\rangle}$
for each root $\dual{\alpha}$ for $\dual{G}$ with respect to $\dual{T}$.
The Lie algebra of the compact part of $\dual{T}(\cpx)$,
denoted by $\dual{T}(\cpx)^\text{u}$ below, 
is then identified with $X^*(T) \otimes \real$ under $\exp$.
We pick a basis $\{ e_0, \ldots, e_g\}$ for $X^*(T)$
that identifies $e_0$ with the similitude character for $\dual{G}$
and write $\{ f_0, \ldots, f_g\}$ for the dual basis
for $X_*(T) \iso X^*(\dual{T})$.
Set $\theta_0 :=0$ and set $\theta := \sum_{j=0}^g \theta_j e_j$;
note that $\theta \in X^*(T)\otimes \real$
so $\exp(\theta)$ lies in $\dual{T}(\cpx)^\text{u}$.
Then $\rho_{X,\lambda}(\Fr_q) = \sqrt{q} \exp(\theta)$.

Let $W_K$ be the Weil group for $K$.
The L-group for $T$ is $\Lgroup{T} = \dual{T}(\cpx)\times W_K$
since $T$ is $K$-split.
Consider the Langlands parameter
\[
\phi : W_K \to \Lgroup{T}
\]
defined by $\phi(\Fr_q) = \rho_{X,\lambda}(\Fr_q) = \sqrt{q} \exp(\theta) \times \Fr_q$.
Let
$
\chi : T(K) \to \cpx^\times
$
be the quasicharacter of $T(K)$ matching $\phi$ under the local Langlands correspondence for algebraic tori \cite{Yu}.
The character
$
\chi^\text{u} \ceq \abs{\ }_K^{-\frac{1}{2}} \otimes \chi
$
corresponds to the unramified Langlands parameter
\[
\phi^\text{u} : W_K \to \Lgroup{T}
\]
defined by
$\phi^\text{u}(\Fr_q) =  \exp(\theta) \times \Fr_q$.

Now pick a Borel subgroup $B\subset G$ over $K$ 
with reductive quotient $T$ and set
\[
\pi_{X,\lambda} \ceq \Ind_{B(K)}^{G(K)} \chi.
\]
Then $\pi_{X,\lambda}$ is an irreducible, unramified principal series representation of $G(K)$.
%
In the same way, the unitary character $\chi^\text{u} : K^\times \to \cpx^\times$ determines the irreducible principal series representation
\[
\pi_{X,\lambda}^\text{u} \ceq \Ind_{B(K)}^{G(K)} \chi^\text{u}.
\]
 This admissible representation $\pi_\lambda^\text{u}$ is unitary and enjoys
\[
\pi_{X,\lambda}^\text{u} = \abs{\ }_K^{-\frac{1}{2}} \otimes \pi_{X,\lambda},
\]
as promised.

Having identified the irreducible principal series representation $\pi_{X,\lambda}$ of $G(K)$ attached to $(X,\lambda)$,
we turn to the L-function $L(s,\pi_{X,\lambda},r_\lambda)$.
For this it will be helpful to go back and say a few words about the representation
$r_{\lambda,\ell} : \GSp(V_\ell X,\ang{\cdot,\cdot}_\lambda) \hookrightarrow \GL(V_\ell X)$.

Let $S$ be a maximal torus in $\GSp(V_\ell X,\ang{\cdot,\cdot}_\lambda)$ containing $\rho_{X,\ell}(\Fr_q)$ and let $S'$ be a maximal torus in $\GL(V_\ell X)$ containing $r_{\lambda,\ell}(S)$.
Let $F_\ell$ be the splitting extension of $S'$ in $\bar\rat_\ell$;
observe that this contains the splitting extension of $\charpoly_{X_0/{\ff_q}}(T)\in \rat[T]$ in $\bar\rat_\ell$.
 Passing from $\rat_\ell$ to $F_\ell$, we may choose bases $\{f_0, f_1, f_2, \ldots , f_g\}$ for $X^*(S)$ and $\{f'_1, f'_2, \ldots , f'_{2g}\}$ for $X^*(S')$ such that the map $X^*(S') \twoheadrightarrow X^*(S)$ induced by the representation $r_{\lambda,\ell}$ is given by
\begin{eqnarray}\label{dr}
X^*(S') &\to& X^*(S) \qquad
f'_j \mapsto f_j,\
f'_{g+j} \mapsto f_0 - f_{g-j+1}, \quad j=1,\ldots, g.
\end{eqnarray}
Note that this determines a basis for $V_\ell X\otimes_{\rat_\ell} F_\ell$.

Passing from $F_\ell$ to $\cpx$ we have now identified a basis for $V_\ell X \otimes_{\rat_\ell} \cpx$ which defines
\[
\GSp(V_\ell X \otimes_{\rat_\ell} \cpx,\ang{\cdot,\cdot}_\lambda)\iso \GSp_{2g}(\cpx) = \dual{G}(\cpx)
\]
inducing $S\otimes_{\rat_\ell} \cpx \iso \dual{T}$ and also gives
\[
\GL(V_\ell X \otimes_{\rat_\ell} \cpx)\iso \GL_{2g}(\cpx).
\]
Now \eqref{rell} extends to
\begin{equation}
\xymatrix{
\Gal({\bar K}/K) \ar[rr]^{\rho_{X}}  \ar[rd]_{\rho_{X,\lambda}} && \GL_{2g}(\cpx)  \\
& \dual{G}(\cpx) \ar[ur]_{r_{\lambda}} & \\
}
\end{equation}
It follows immediately that
\[
L(s,\pi_{X,\lambda},r_\lambda)
= \prod_{i=1}^{2g} \frac{1}{1-\tau_i q^{-s}}
= \prod_{i=1}^{2g} \frac{q^{s}}{q^s-\tau_i}
= \frac{(q^s)^{2g}}{\charpoly_{X_0/{\ff_q}}(q^s)}
= L(s,\rho_{X}),
\]
concluding the proof of Proposition~\ref{prop:piX}.
\end{proof}

\subsection{R-groups}\label{sec:R}

The irreducible representation $\pi_{X,\lambda}$ of $G(K)$ appearing in Proposition~\ref{prop:piX} is obtained by parabolic induction from an unramified quasicharacter of a split maximal torus $T(K)$. In Section~\ref{sec:res} we will use the restriction of this representation to the derived group $G\der(K) = \Spin_{2g+1}(K)$ of $G(K) = \GSpin_{2g+1}(K)$ to study $X$. While the resulting representation of $G\der(K)$ is again an unramified principal series representation, it need not be irreducible; in fact, we will glean information about $X$ from the components of this representation of $G\der(K)$. With this application in mind, here we review some basic facts about reducible principal series representations of $G\der(K)$.

As in the proof of Propositon~\ref{prop:piX}, let $B$ be a Borel subgroup of $G$ with reductive quotient $T$, a split maximal torus in $G$.
Set $B\der = G\der \cap B$. This is a Borel subgroup of $G\der$ with reductive quotient $T\der = T \cap G\der$, a split maximal torus in $G\der$.
Let $\sigma$ be a character of $T\der(K)$. The component structure of the admissible representation $\Ind_{B\der(K)}^{G\der(K)} \sigma$ is governed by the commuting algebra $\End(\Ind_{B\der(K)}^{G\der(K)} \sigma)$ which, in turn, is given by the group algebra $\cpx[R(\sigma)]$, where $R(\sigma)$ is the Knapp-Stein R-group; see \cite{Keys:decomposition}*{Introduction} for a summary and references to original sources, including \cite{Silberger}.

The Knapp-Stein R-group $R(\sigma)$ is determined as follows, as explained in \cite{Keys:decomposition}*{\S 3}. Let $R$ be the root system for $G$ with respect to $T$ and let $W$ be the corresponding Weyl group for $G$. The root system for $G\der$ may be identified with $R$; see Table~\ref{table:rootdata}.
Set $W_\sigma = \{ w\in W \tq \,^w\sigma = \sigma\}$. For each root $\alpha\in R$, let $\sigma_\alpha$ be the restriction of $\sigma$ to the rank-1 subtorus $T_\alpha\subseteq T$.
 Consider the root system $R_\sigma = \{ \alpha \in R \tq \sigma_\alpha =1 \}$. Then $R(\sigma) = \{ w\in W_\sigma \tq w(R_\sigma) = R_\sigma\}$.  The exact sequence
\[
1\to W_\sigma^\circ \to W_\sigma \to R(\sigma) \to 1
\]
determines $R(\sigma)$, with $W_\sigma^\circ \ceq \{ w_\alpha \tq \alpha \in R_\sigma \}$, the Weyl group of the root system $R_\sigma$; see \cite{Keys:decomposition}*{\S 3}. 

We will need the following alternate characterization of $R(\sigma)$.
Let $s\in \dual{T}\ad(\cpx)$ be the semisimple element of $\dual{G}\ad(\cpx)$ corresponding to the character $\sigma$ of $T\der(K)$.
By \cite{Steinberg}*{\S 3.5, Prop. 4}
 (see also \cite{Humphreys}*{\S 2.2, Theorem}),
 $Z_{\dual{G}\ad(\cpx)}(s)$ is a reductive group with root system
$\dual{R}_s \ceq \{ \dual{\alpha}\in \dual{R} \tq \dual{\alpha}(s) =1\}$.
The bijection between $R$ and $\dual{R}$ which comes with the root datum for $G$ restricts to a bijection between $R_\sigma$ and $\dual{R}_s$.
Moreover, by \cite{Steinberg}*{\S 3.5, Prop. 4} again,
the component group of the reductive group  $Z_{\dual{G}\ad(\cpx)}(s)$ is 
$W_s/W_s^\circ$ where $W_s^\circ$ is the Weyl group for the root system $\dual{R}_s$ and $W_s = \{ w\in W \tq w(s)  = s\}$:
\[
1\to W_s^\circ \to W_s \to \pi_0(Z_{\dual{G}\ad(\cpx)}(s)) \to 1.
\]
Here we have identified the Weyl group $W$ for $R$ with the Weyl group for $\dual{R}$. Under that identification, $W_s = W_\sigma$ and $W_\sigma^\circ = W_s^\circ$, so
\[
R(\sigma) \iso \pi_0(Z_{\dual{G}\ad(\cpx)}(s)),
\]
canonically. 

\subsection{Component group calculations}

Now we calculate the group $\pi_0(Z_{\dual{G}\ad(\cpx)}(s))$.

\begin{table}[htdp]
\caption{Based root data for $\GSpin_{2g+1}$, $\Spin_{2g+1}$ and $\SO_{2g+1}$.}
\begin{center}
{\smaller\smaller\smaller\smaller\smaller
\[
\begin{array}{|rl |c| rl |c| rl |}
\hline
\text{semisimple,}&&& &  && & \\
& \hskip-20pt\text{simply connected} && \text{\bf Type:}& B_g  & & \text{adjoint} &\\
\hline
&&&&&&&\\
G\der =& \Spin_{2g+1} &\rightarrowtail & G =& \GSpin_{2g+1} &\twoheadrightarrow& G\ad =& \SO_{2g+1} \\
T\der =& \Gm^{g} &\rightarrowtail & T =& \Gm^{g+1} & \twoheadrightarrow & T\ad =& \Gm^g \\
Z(G\der) =& \mu_2 &\rightarrowtail & Z(G) =& \Gm &\twoheadrightarrow & Z(G\ad) =& 1 \\
&&&&&&&\\
\hline
&&&&&&&\\
X^*(T\der) =& \langle e_1, \ldots , e_g \rangle  &0 \mapsfrom e_0 & X^*(T) = & \langle e_0, e_1, \ldots , e_g \rangle  &\leftarrowtail  &  X^*(T\ad) =&  \langle \alpha_1, \ldots, \alpha_g\rangle \\
&&&&&&&\\
\hline
& &  &  &&  & &   \\
R\der := & R(G\der,T\der)  & & R:=  & R(G,T) && R\ad := & R(G\ad,T\ad)  \\
= & \langle \alpha_1, \ldots, \alpha_g \rangle  & & =  & \langle \alpha_1, \ldots, \alpha_g\rangle && = & \langle \alpha_1, \ldots, \alpha_g\rangle  \\
&&&&&&&\\
\alpha_1 = &e_1-e_2 & &  \alpha_1 = & e_1-e_2 && &  \\
\alpha_2 =  & e_2-e_3&& \alpha_2 = & e_2-e_3 && &  \\
& \vdots &&  & \vdots &&  &   \\
\alpha_{g-1} = & e_{g-1}-e_g& & \alpha_{g-1} = & e_{g-1}-e_g &&&  \\
\alpha_g = & e_g & &\alpha_g = & e_g & &&   \\
&&&&&&&\\
\hline
&&&&&&&\\
X^*(T\der)/\langle R\der \rangle =& \integ/2\integ&& X^*(T)/\langle R \rangle =& \integ && X^*(T\ad) =& \langle R\ad \rangle  \\
&&&&&&&\\
\text{weight lattice}  =& X^*(T\der)  && \text{weight lattice} =& X^*(T)  && \frac{X^*(T\ad)}{\text{weight lattice}} = & \integ/2\integ \\
&&&&&&&\\
\hline
\hline
&  & &  &&& \text{semisimple,} & \\
\text{adjoint} &  && \text{\bf Type:}& C_g && &\hskip-20pt\text{simply connected}  \\
\hline
&&&& &&&\\
\dual{G}\ad =& \PGSp_{2g} &\twoheadleftarrow& \dual{G} =& \GSp_{2g} & \leftarrowtail& \dual{G}\der =& \Sp_{2g}\\
\dual{T}\ad = & \Gm^{g} & \twoheadleftarrow  & \dual{T} = & \Gm^{g+1} & \leftarrowtail & \dual{T}\der = & \Gm^{g}\\
Z(\dual{G}\ad) =& 1 &\twoheadleftarrow & Z(\dual{G}) =& \Gm & \leftarrowtail & Z(\dual{G}\der) =& \mu_2 \\
&&&&&&&\\
\hline
&&&&&&&\\
 X^*(\dual{T}\ad) =& \langle \dual{\alpha}_1, \ldots, \dual{\alpha}_g\rangle & \rightarrowtail & X^*(\dual{T}) = & \langle f_0, f_1, \ldots , f_g \rangle  & f_0 \mapsto 0 & X^*(\dual{T}\der) =& \langle f_1, \ldots , f_g \rangle\\
&&&&&&&\\
\hline
&&  & &     &   & & \\
  \dual{R}\ad := & R(\dual{G}\ad,\dual{T}\ad) & & \dual{R} :=  &R(\dual{G},\dual{T}) & & \dual{R}\der := & R(\dual{G}\der,\dual{T}\der) \\
 = & \langle \dual{\alpha}_1, \ldots, \dual{\alpha}_g\rangle & & = & \langle \dual{\alpha}_1, \ldots, \dual{\alpha}_g\rangle &&= & \langle \dual{\alpha}'_1, \ldots, \dual{\alpha}'_g\rangle\\
 &&&&&&&\\
 & & & \dual{\alpha}_1= & f_1-f_2 & & \dual{\alpha}_1' = & f_1-f_2 \\
&& & \dual{\alpha}_2 =& f_2-f_3 & &\dual{\alpha}_2'= &  f_2-f_3  \\
  &&  && & \vdots& & \vdots \\
  && & \dual{\alpha}_{g-1}=&  f_{g-1}-f_g && \dual{\alpha}_{g-1}'=&  f_{g-1}-f_g \\
  & & & \dual{\alpha}_g= & 2f_g - f_0 & & \dual{\alpha}_g' = & 2f_g \\
&&&&&&&\\
\hline
&&&&&&&\\
X^*(\dual{T}\ad) =& \langle \dual{R}\ad \rangle && X^*(\dual{T}) /\langle \dual{R} \rangle =& \integ && X^*(\dual{T}\der)/\langle \dual{R}\der \rangle =& \integ/2\integ \\
&&&&&&&\\
\frac{X^*(\dual{T}\ad)}{\text{weight lattice}} = & \integ/2\integ  && \text{weight lattice} = &\langle \dual{R} \rangle &&\text{weight lattice}  =&  X^*(\dualgroup{T}\der) \\
&&&&&&&\\
\hline
\end{array}
\]
}
\end{center}
\label{table:rootdata}
\end{table}

\begin{proposition}\label{prop:Sp}
Suppose $t\in \GSp_{2g}(\cpx)$ is semisimple and all eigenvalues have complex modulus $1$.
Let $s\in \PGSp_{2g}(\cpx)$ be the image of $t$ under $\GSp_{2g}(\cpx) \to \PGSp_{2g}(\cpx)$.
Then 
\[
\pi_0(Z_{\PGSp_{2g}(\cpx)}(s)) \iso \integ/2\integ
\]
if and only if the characteristic polynomial of $r_\lambda(t)$ is even; 
otherwise, $\pi_0(Z_{\PGSp_{2g}(\cpx)}(s))$ is trivial.
\end{proposition}

\begin{proof}
Using the notation in the proof of Proposition~\ref{prop:piX}, 
pick $x\in X^*(T)\otimes \real$ with $\exp(x) =t$;
of course, $x$ is not uniquely determined by $t$, as the kernel of $\exp : X^*(T)\otimes \real \to  \dual{T}(\cpx)$ is the weight lattice for $T$, which, in this case, is the character lattice $X^*(T)$ itself; see Table~\ref{table:rootdata}.
Let $v \in X^*(T\der) \otimes\real$ be the image of $x$ under the map $X^*(T) \otimes\real \to X^*(T\der) \otimes\real$ induced from $X^*(T) \to X^*(T\der)$; see Table~\ref{table:rootdata}. 
Note that $\exp(v) = s$ where now $\exp$ refers to the map $\exp: X^*(T\der) \to \dual{T}(\cpx)$ defined as above.
Using this map we may identify $\Lie \dual{T}\ad(\cpx)$ with $X^*(T\der) \otimes \cpx$; under this identification, the Lie algebra of the compact subtorus of $\dual{T}\ad(\cpx)$ may be identified with $X^*(T\der) \otimes \real$, henceforth denoted by $V$.

Let $R\der$ be the root system for $G\der$ and let  $\langle R\der\rangle$ be the lattice generated by $R\der$. By \cite{Reeder}*{\S2.2}, 
\begin{equation}\label{MR}
\pi_0(Z_{\dual{G}\ad(\cpx)}(s)) \iso \{ \gamma \in X^*(T\der)/\langle R\der \rangle \tq \gamma(v) =v \},
\end{equation}
for a canonical action of $X^*(T\der)/\langle R\der \rangle$ on $V$,
which we will use to calculate $\pi_0(Z_{\dual{G}\ad(\cpx)}(s))$.
Even before describing this action, however, we remark that \eqref{MR}, together with the calculation of $X^*(T\der)/\langle R\der \rangle$ in Table~\ref{table:rootdata}, already gives us good information about $\pi_0(Z_{\dual{G}\ad(\cpx)}(s))$: this component group is trivial or $\integ/2\integ$, and in particular, abelian.

In order to describe the action of $X^*(T\der)/\langle R\der \rangle$ on $V$ 
and calculate the right hand side of \eqref{MR},
 we must introduce yet more notation.
Adapting \cite{Bourbaki}*{VI,\S 2}, let $W_\text{aff} \ceq \langle R\der \rangle \rtimes W$ be the affine Weyl group for $\dual{G}\ad$
 and let $W_\text{ext} \ceq X^*(T\der) \rtimes W$ be the extended affine Weyl group for $\dual{G}\ad$. 
(Here we use the coincidence of the weight lattice for $G\der$ with the character lattice for $G\der$.)
Then $W_\text{ext}$ as a semidirect product of the Coxeter group $W_\text{aff}$ by $X^*(T\der)/\langle R\der \rangle$.
\begin{equation}\label{ext}
1 \to  W_\text{aff} \to W_\text{ext} \to X^*(T\der)/\langle R\der \rangle \to 1
\end{equation}
The quotient $X^*(T\der)/\langle R\der \rangle$ coincides with the fundamental group $\pi_1(\dual{G}\ad)$ of $\dual{G}\ad$ (see \cite{Steinberg:Yale}*{p. 45} for a table of these finite abelian groups by type).  By \cite{Bourbaki}*{VI,\S 2.4, Cor.}, the minuscule coweights for $\dual{G}\ad$ determine a set of representatives for $X^*(T\der)/\langle R\der \rangle$.
The basis in Table~\ref{table:rootdata} for the root system $\dual{R}\ad$  determines the alcove
\[
C \ceq \{ v\in V \tq \langle \dual{\alpha}_i, v\rangle > 0 , 0\leq i \leq n\}
\]
in $V$, where $\dual{\alpha}_0$ is the affine root for which $1-\dual{\alpha}_0$ is the longest root with respect to the given basis for $\dual{R}\ad$;
 see \cite{Bourbaki}*{VI, \S 2.3}. The closure ${\bar C}$ of  $C$ is a fundamental domain for the action of $W_\text{aff}$ on $V$.
The affine Weyl group $W_\text{aff}$ acts freely and transitively on the set of alcoves in $V$. The extended affine Weyl group $W_\text{ext}$ acts transitively on the set of alcoves, but generally not freely.
Since minuscule coweights for $\dual{G}\ad$ determine a set of representatives for $X^*(T\der)/\langle R\der \rangle$, and since each such coweight may be identified with a vertex of ${\bar C}$ (not all vertices arise this way), we have
\begin{equation}\label{Omega}
\{ w \in W_\text{ext} \tq w(C) = C \} \iso X^*(T\der)/\langle R\der \rangle,
\end{equation}
canonically.
This describes the action of $X^*(T\der)/\langle R\der \rangle$ on $V$.

The calculation of $\{ \gamma \in X^*(T\der)/\langle R\der \rangle \tq \gamma(v) =v \}$ now follows easily.
Let $\{ \varpi_1, \ldots , \varpi_g\}$ be the basis of weights for $X^*(T\der)$ dual to the basis $\dual{R}\ad = \{ \dual{\alpha}_1,\ldots, \dual{\alpha}_g\}$ for $X^*(\dual{T}\ad) = X_*(T\der)$; set $\varpi_0 = 0$. 
The closure ${\bar C}$ of the alcove $C$ is the convex hull of the vertices $\{ v_0, v_1, \ldots , v_g\}$ defined by $v_j = \frac{1}{b_j}\varpi_j$, where $b_0 = 1$ and the other integers $b_j$ are determined by the longest root in $\dual{R}\ad$ according to $\dual{\alpha} = \sum_{j=1}^g b_j \dual{\alpha}_j$.
In the case at hand the longest root is
$\dual{\alpha} = 2 \dual{\alpha}_1 + 2  \dual{\alpha}_2 + \cdots + 2 \dual{\alpha}_{g-1} + \dual{\alpha}_g$ so
$b_1 = 2$, $\ldots$ , $b_{g-1}=2$, $b_g = 1$.
Note that exactly two vertices in $\{ v_0, v_1, \ldots, v_g\}$ are hyperspecial: $v_0$ and $v_g$.
Since $W_\text{ext}$ acts transitively on the alcoves in $V$ and since $\exp : V\to \dual{T}\ad(\cpx)$ is $W_\text{ext}$-invariant, we may now suppose $v\in {\bar C}$.
Express $v\in V$ in the basis of weights for $X^*(T\der)$:
\begin{equation}\label{v}
v = \sum_{j=1}^g x_j \varpi_j;
\end{equation}
note that the coefficients in this expansion are precisely the root values $x_j = \dual{\alpha}_j(v)$. Then $v\in {\bar C}$ exactly means $x_j \geq 0$. Set $b_0=1$ and define $x_0\geq 0$ so that $\sum_{j=0} b_j x_j =1$;
in other words
\[
v = \sum_{j=0}^g x_j \varpi_j, \qquad x_0 + 2 x_1 + \cdots + 2x_{g-1} + x_g =1.
\]
Under the isomorphism \eqref{Omega}, the non-trivial element of $X^*(T\der)/\langle R\der \rangle$ corresponds to $\rho \in W_\text{ext}$ defined by
$v_j \mapsto v_{g-j}$ for $j=0, \ldots, g$. In terms of the fundamental weights $\{ \varpi_0, \varpi_1, \ldots , \varpi_g\}$ this affine transformation is defined by $\varpi_j \mapsto \varpi_{g-j}$ for $j=0, \ldots , g$. Thus, $\{ \gamma \in X^*(T\der)/\langle R\der \rangle \tq \gamma(v) =v \}$ is non-trivial if and only if $\rho(v) =v$, which is to say,
\begin{equation}\label{x}
x_j  \geq  0, \quad j = 1, \ldots , g
\end{equation}
and
\begin{align*}
x_1 + \cdots + x_{g-1} + x_g & = \frac{1}{2}\\
x_j &= x_{g-j},\quad j =1, \ldots, g-1.
\end{align*}

It only remains to translate the conditions above into conditions on the eigenvalues of $t \in G(\cpx)$. 
To do that we pass from root values $x_j= \langle \dual{\alpha}_j, x\rangle$ to character values $y_j \ceq \langle f_j, x\rangle$.
Again using Table~\ref{table:rootdata} we see that the conditions above
are equivalent to
\begin{equation}\label{y1}
y_1 \geq y_2 \geq \cdots \geq y_g \geq \frac{1}{2}y_0
\end{equation}
and
\begin{align*}
y_1 +  y_g & = \frac{1}{2} + y_0\\
y_j - y_{j+1} &= y_{g-j} - y_{g-j+1},\quad  j =1, \ldots, g-1.
\end{align*}
When combined, these last two conditions take a very simple form:
\begin{equation}\label{y2}
y_0 - y_j = \frac{1}{2} + y_{g-j+1},\quad  j =1, \ldots, g-1.
\end{equation}

Finally, we calculate the characteristic polynomial of $r_\lambda(t)$.
Observe that $r_\lambda(t)= r_\lambda(\exp(x)) = \exp(dr_\lambda(x))$,
where $dr_\lambda :X^*(T) \to X^*(\Gm^{2g})$ is given by \eqref{dr}.
Set $t_j = e^{2\pi i y_j}$ for $j=0, \ldots , g$. Then constraint \eqref{y2} is equivalent to
\begin{equation}\label{t}
t_0 t_j^{-1} = -t_{g-j+1},\quad  j =1, \ldots, g-1.
\end{equation}
The characteristic polynomial of $r_\lambda(t)$ is
\begin{equation}\label{c}
\charpoly_{r_\lambda(t)}(T) \ceq \prod_{j=1}^g (T- t_j)\prod_{j=1}^g (T- t_0t_j^{-1}).
\end{equation}
When combined with \eqref{t}, it is clear that $\charpoly_{r_\lambda(t)}(T)$ is even:
\begin{eqnarray*}
\charpoly_{r_\lambda(t)}(T)
&=& \prod_{j=1}^g (T- t_j)\prod_{j=1}^g (T+ t_{g-j+1}),\quad \eqref{t}\\
&=& \prod_{j=1}^g (T- t_j)\prod_{i=1}^g (T+ t_{i}), \quad j \mapsto g-j+1\\
&=& \prod_{j=1}^g (T^2- t_j^2).\\
\end{eqnarray*}

We have now seen that if $\pi_0(Z_{\PGSp_{2g}(\cpx)}(s))$ is non-trivial, then $\charpoly_{r_\lambda(t)}(T)$ is even.
To see the converse, suppose $\charpoly_{r_\lambda(t)}(T)$ \eqref{c} is even.
Without loss of generality, we may assume the similitude factor $t_0$ is trivial.
Then, after relabelling if necessary, the \emph{symplectic} characteristic polynomial $\charpoly_{r_\lambda(t)}(T)$ is even
if and only if it takes the form $\charpoly_{r_\lambda(t)}(T) = \prod_{j=1}^{g} (T^2- r_j^2)$, with $r_j^{-1} = -r_{\sigma(j)}$
for some permutation $\sigma$ of  $\{ 1, \ldots , g\}$.
Since the roots are the eigenvalues of $t$, which are unitary by hypothesis, we can order them by angular components, as in \eqref{y1},
while replacing $\sigma$ with the permutation $j \mapsto g-j+1$,
thus bringing us back to \eqref{y2}.
This concludes the proof of Proposition~\ref{prop:Sp}.
 \end{proof}

\subsection{Restriction to the derived group}\label{sec:res}

In this section we show how to recognize when $X/K$ is even through a simple property of the admissible representation $\pi_{X,\lambda}$ of $G(K)$.

\begin{theorem}\label{thm:red}
Let $X/K$ be an abelian variety of dimension $g$ with good reduction
and let $\lambda$ be a polarization on $X$.
The restriction of $\pi_{X,\lambda}$ from $\GSpin_{2g+1}(K)$ to $\Spin_{2g+1}(K)$ is reducible if and only if $X$ is even.
\end{theorem}

\begin{proof}
With reference to notation from the proof of Proposition~\ref{prop:piX}, set $t = \exp(\theta)$ and let $s\in \dual{T}\ad$ be the image of $t$ under $\dual{T}\to \dual{T}\ad$.
The restriction of $\pi_{X,\lambda}$ from $G(K)$ to $G\der(K)$ decomposes into irreducible representations indexed by the component group $\pi_0(Z_{\dual{G}\ad(\cpx)}(s))$.
Indeed, the irreducible representations of $G\der(K)$ that arise in this way
are precisely the irreducible representations appearing in
$\Ind_{B\der(K)}^{G\der(K)}\chi\der$,
where $B\der(K)$ is a Borel containing $T\der(K)$ and $\chi\der$ is the
unramified quasicharacter of $T\der(K)$ corresponding to $t\ad \in \dual{T}\ad(\cpx)$.
The R-group for this unramified principal series representation is
$
\pi_0(Z_{\dual{G}\ad}(s)).
$
By Proposition~\ref{prop:Sp}, this group is either trivial or a group of order $2$, so either $\pi_{X,\lambda} \rest{G\der(K)}$ is irreducible or contains two irreducible admissible representations; also by Proposition~\ref{prop:Sp}, the latter case occurs if and only if the characteristic polynomial $\charpoly_{X_0/{\ff_q}}(T)$ is even,
in which case $X/K$ itself is even (Proposition \ref{lemevenlocal}).
\end{proof}

\subsection{L-packet interpretation}

In this section we show how to recognize even abelian varieties
over local fields through associated L-packets.

As discussed in Section~\ref{sec:polar}, every polarized abelian variety $(X,\lambda)$ over $K$ determines an $\ell$-adic Galois representation $\rho_{X,\lambda,\ell} : \Gal({\bar K}/K) \to \GSp(V_\ell X, \ang{\cdot,\cdot}_\lambda)$.
Let $W'_K$ be the Weil-Deligne group for $K$ \cite{Tate:NTB}*{\S 4.1}.
Let $\phi_{X,\lambda,\ell} : W'_K \to \Gal({\bar K}/K) \to \GSp(V_\ell X, \ang{\cdot,\cdot}_\lambda)$
 be the Weil-Deligne homomorphism obtained by applying
\cite{Deligne:L}*{Thm 8.2} to $\rho_{X,\lambda,\ell}$.
 We note that
$\Lgroup{G} = \dual{G}(\cpx)\rtimes W_K = \GSp_{2g}(\cpx)\times W_K$.
Let
\[
\phi_{X,\lambda} : W'_K \to \Gal({\bar K}/K) \to \Lgroup{G}
\]
be the admissible homomorphism determined by $\phi_{X,\lambda,\ell}$ and the basis for $V_\ell X \otimes_{\rat_\ell} \cpx$ identified in the proof of Proposition~\ref{prop:piX}.
The equivalence class of the admissible homomorphism $\phi_{X,\lambda}$ is the Langlands parameter for the polarized abelian variety $(X,\lambda)$ over $K$. We remark that this recipe is valid for all polarized abelian varieties over $K$, not just those of good reduction.
But here we are interested in the case when $X$ has good reduction, in which case $\rho_{X,\lambda}$ is unramified in the strongest sense: the local monodromy operator for the Langlands parameter $\phi_{X,\lambda}$ is trivial ($\phi_{X,\lambda}$ factors through $W'_K \to W_K$) and $\phi_{X,\lambda}$ is trivial on the inertia subgroup $I_K$ of $W_K$.

Although the full local Langlands correspondence for $G= \GSpin_{2g+1}$ is not yet known,
the part which pertains to unramified principal series representations is,
allowing us to consider the L-packet $\Pi_{X,\lambda}$
for the Langlands parameter $\phi_{X,\lambda}$.
Indeed, we have seen that this L-packet contains the equivalence
class of $\pi_{X,\lambda}$, only.

Theorem~\ref{thm:red} shows that we can detect when $X$ is $K$-isogenous to its twist over the quadratic unramified extension of $K$ by restricting $\pi_{X,\lambda}$ from $G(K)$ to $G\der(K)$.
On the Langlands parameter side, this restriction corresponds to post-composing $\phi_{X,\lambda}$ with $\Lgroup{G} \to \Lgroup{G}\ad$.
Let $\phi_{X,\lambda}^\text{der}$ be the Langlands parameter for $G\der/K$ defined by the diagram below and let $\Pi_{X,\lambda}^\text{der}$ be the corresponding L-packet.
\begin{equation}\label{phiad}
\xymatrix{
W'_K \ar[rr]^{\phi_{X,\lambda}}  \ar@{.>}[dr]_{\phi_{X,\lambda}^\text{der}} && \Lgroup{G}  \ar@{->>}[dl] \\
& \Lgroup{G}\ad & \\
}
\end{equation}

\begin{corollary}\label{cor:L}
Let $X/K$ be an abelian variety of dimension $g$ with good reduction
and let $\lambda$ be a polarization on $X$.
The L-packet $\Pi_{X,\lambda}^\text{der}$ for $\Spin_{2g+1}(K)$ has cardinality $2$ exactly when $X$ is even; otherwise, it has cardinality $1$.
\end{corollary}

\begin{proof}
This follows directly from the fact that the R-group for any
representation in the restriction of $\pi_{X,\lambda}$ to $G\der(K)$
coincides with the Langlands component group attached to $\phi_{X,\lambda}^\text{der}$.
(See \cite{Ban-Goldberg:R-groups} for more instances of this coincidence.)
Namely, equivalence classes of representations that live in $\Pi_{X,\lambda}^\text{der}$
are parameterized by irreducible representations of the group
\[
\mathcal{S}_{\phi_{X,\lambda}^\text{der}} \ceq Z_{\dual{G}\ad}(\phi_{X,\lambda}^\text{der})/ Z_{\dual{G}\ad}(\phi_{X,\lambda}^\text{der})^0\ (Z\dual{G}\ad)^{W_K}.
\]
Since $G\der$ is $K$-split, the action of $W_K$ on $\dual{G}\ad$ is trivial, and
since $\phi_{X,\lambda}^\text{der}$ is unramified, $Z_{\dual{G}\ad}(\phi_{X,\lambda}^\text{der}) = Z_{\dual{G}\ad}(t\ad)$, where $t\ad = \phi_{X,\lambda}^\text{der}(\Fr_q)$;
thus,
\[
\mathcal{S}_{\phi_{X,\lambda}^\text{der}} = \pi_0(Z_{\dual{G}\ad}(t\ad)),
\]
which is precisely the R-group for $\pi_{X,\lambda}\vert_{G\der(K)}$, calculated in Theorem~\ref{thm:red}.
\end{proof}

\section{Concluding remarks}

It is natural to ask how the story above extends to include abelian varieties $X$
over local fields which do not necessarily have good reduction, keeping track of the
relation between the $\ell$-adic Tate module $T_\ell X$ and the associated
Weil-Deligne representations, generalizing \cite{Rohrlich},
and the corresponding L-packets.
For this it would be helpful to know the full local Langlands correspondence
for $\GSpin_{2g+1}(K)$, not just the part which pertains to unramified
principal series representations.
 Since the full local Langlands correspondence for $\GSpin_{2g+1}(K)$
is almost certainly within reach by an adaptation of Arthur's work \cite{Arthur:endoscopic_classification} on the endoscopic classification of representations, following \cite{Arthur:GSp4},
we have, for the moment, postponed looking into such questions until Arthur's ideas have been adapted to general spin groups.

At the heart of this note we have used a very simple instance of what is, according to a conjecture of Arthur \cite{Arthur:unipotent-conjectures}, a very general phenomenon: the coincidence of Knapp-Stein R-groups with the component groups attached to Langlands parameters, sometimes known as Arthur R-groups, as in \cite{Ban-Zhang:R-groups}. While most known cases of this coincidence appear or are summarized in \cite{Ban-Goldberg:R-groups}, as remarked at the end of the introduction to that paper, there is work remaining for general spin groups.

When some of these missing pieces are available, we intend to use the local results in this note to explore the connection between abelian varieties over number fields and global L-packets of automorphic representations of spin groups and general spin groups, generalizing the results of \cite{Anand-Prasad:SL2}*{\S 2}.

\end{document}